\documentclass[12pt]{article}
\usepackage{amsmath}
\usepackage{amsthm}
\usepackage{amsfonts}
\usepackage{exscale}
\usepackage{enumerate}
\usepackage{hyperref}
\catcode`\^^M=10  
\DeclareMathOperator{\degre}{degree}
\DeclareMathOperator{\norm}{Norm}

\newcommand{\abs}[1]{{\left|#1\right|}}
\newcommand{\pp}{{\mathfrak{p}}}

\newtheorem{thm}[equation]{Theorem}

\newtheorem{lemma}[equation]{Lemma}
\newtheorem{corollary}[equation]{Corollary}

\theoremstyle{definition}
\newtheorem{remark}[equation]{Remark}

\newtheorem*{hyp*}{\sc Hypothesis $(**)$}
\theoremstyle{definition}

\newcommand{\Q}{{\mathbf{Q}}}
\newcommand{\Z}{{\mathbf{Z}}}
\DeclareMathOperator{\End}{End}
\DeclareMathOperator{\Hom}{Hom}
\DeclareMathOperator{\id}{id}
\DeclareMathOperator{\inff}{inf}
\DeclareMathOperator{\Mat}{Mat}

\newcommand{\mathdef}{\emph}

\def\bv{\setbox0\vbox\bgroup\hsize4.18in}
\def\ev{\egroup\hskip1.in\box0}

\title{Norms in Central Simple Algebras}
\author{Daniel Goldstein \and Murray Schacher}

\begin{document}
\maketitle
\bigskip
\centerline{To Robert Steinberg, a cherished teacher, colleague, and friend}
\bigskip
\begin{abstract}
Let $A$ be a central simple algebra over a number field $K$.  In this
note we study the question of which integers of $K$ are reduced norms
of integers of $A$.  We prove that if $K$ contains an integer that is
the reduced norm of an element of $A$ but not the reduced norm of an
integer of $A$, then $A$ is a totally definite quaternion algebra over
a totally real field (i.e.\ $A$ fails the \mathdef{Eichler condition}).
\end{abstract}
\par

\section{Introduction}
Let $A$ be a central simple algebra over a number field $K$.
Write $\norm(\cdot)$ for the
reduced norm from $A$ to $K$. If $x$ is an integer in $A$, 
then clearly $\norm(x)$
lies in $R$, the ring of integers of $K$. It is also clear that $x$
must be positive at the real primes of $K$ at which $A$ is ramified. Suppose
that $m \in R$ satisfies this property, and so $m$ is a norm from $A$ (see Theorem \ref{hms}).
If $m$  is not the reduced norm of an
integer of $A$, we call $m$ an \mathdef{outlier for $A$} (this terminology is
not standard).

\par 

The main result of this paper (combining Theorem~\ref{thm:A} and
Lemmas~\ref{totdef} and~\ref{finite})
is that if $K$ contains an outlier for
$A$, then $K$ is totally real, $A$ is a quaternion algebra over $K$,
and $A$ is totally definite. (One says in this case that $A$ fails the 
Eichler condition).

\par

We also prove a theorem of Deligne in Section~\ref{sec:deligne}
(because we couldn't find a proof in
the literature), which states that if $n\ge2$ is an integer, and $E_1,\cdots,E_n$ and $F_1,\cdots,F_n$ are supersingular elliptic curves defined over 
an algebraic closure of the finite field $GF(p)$, the field of $p$ elements, then
$$E_1\times\cdots\times E_n\cong F_1\times\cdots\times F_n.
$$

\par 

The main ingredient is Eichler's theorem on the uniqueness of a
maximal order in a csa in which Eichler's condition holds. 
We also exploit the known fact that the endomorphism algebra of such an
$E_i$ is a maximal order in the quaternion algebra $A_p$ over the rational field $\Q$ ramified
at $p$ and $\infty$ and unramified everywhere else (and every maximal order arises in this context).
Using this
connection also allows one to interpret outliers in $\Q$ for $A_p$
as positive integers $m$ for which no supersingular
elliptic curve defined over the algebraic closure of $GF(p)$ 
has an endomorphism of degree $m$.

\section{Notation and Terminology}
\par 
Throughout this paper, $K$ is a number field, $R$ its ring of
integers, and $A$  is a \mathdef{central simple algebra}  
over $K$.  By definition, $A$ is a finite-dimensional algebra over
$K$, the center of $A$ is equal to $K,$ and $A$ has no nonzero
$2$-sided ideals.
Equivalently, $A\otimes_K \bar{K}$ is isomorphic to the matrix algebra

$M_n(\bar{K})$,
where $\bar{K}$ denotes an algebraic closure of $K$.
For basic facts about central simple algebras see \cite{pierce}.
\par

The positive integer $n$ is the \mathdef{degree} of $A$.
A central division algebra $D$ is a central simple algebra, as is
$M_k(D)$ for any $k$, and conversely every 
central simple algebra over $K$ is of this form
by Wedderburn's Theorem~\cite[Chapter IX, \S1, Prop. 2]{weil}.
\par

A division algebra of degree $n = 2$ is called a \mathdef{quaternion algebra}.  
\par

If $L$ is a field extension of $K$, then $A\otimes_K L$ is a central
simple algebra over $L$. If $A\otimes_K L$ is isomorphic to $M_n(L)$
then $L$ is said to \mathdef{split} $A$.

\par
Let $M$ denote the set of places of the number field $K$.
For each place $v\in M$,  $A_v := A \otimes_K K_v$
is a central simple algebra over the completion $K_v$.  
By Wedderburn's theorem, it is a ring of matrices over a local division
ring $D_v$ central over $K_v$.  We set $n_v = \degre (D_v)$;
$n_v$ is called the \mathdef{local degree}.
$A$ is said to be \mathdef{split} at $v$ if $K_v$ splits $A$ ($n_v=1$);
otherwise it is \mathdef{ramified} at $v$ ($n_v>1$).
A key fact is that a central simple algebra over $K$
splits at all but finitely many places $v$ of $K$.
\par
We have the following splitting criterion (see \cite{reiner}):

\begin{lemma}\label{splitting}\ Let $A$ be a central simple algebra
 over the number field $K$. The finite extension $L$ of $K$ splits $A$
 if and only if, for each place $v$
 of $K$ and for each extension $w$ of $v$ to $L$, the local dimension 
 $[L_w : K_v]$ is a multiple of the local degree $n_v$.
\end{lemma}

\par
Note that to determine, using Lemma~\ref{splitting},
whether a given finite extension $L$ over $K$ splits $A$, 
it is enough to check the stated condition
at the finite set of places $v$ of $K$ where $A$ is ramified.

\par
The notion of reduced norms in a central simple algebra $A$ is bound up with
the two notions of subfields and splitting fields.  A field extension $L$ of $K$
is a subfield of $A$ if $L$ embeds in $A$; a maximal subfield of $A$
is a maximal such.  
All maximal subfields of $A$ have dimension $n=\degre(A)$ over~$K$.  
A maximal subfield of $A$ is  a splitting field for $A$, 
and conversely every $n$-dimensional splitting field for $A$ embeds in $A$
as a maximal subfield \cite[Chapter 1, Section 7]{reiner}.  When $A$
is a quaternion algebra, this translates as: maximal subfields of $A$ are quadratic over $K$, and
quadratic splitting fields of $A$ embed in $A$.  We will use this association
later on.

\par If $A$ is a central simple algebra over $K$, $L$ a maximal subfield of $A$, and
$x \in L$, then the \mathdef{reduced norm} $\norm(x)$ is the ordinary field norm
from $L$ to $K$.  This notion is independent of the choice of $L$, or of the embedding
of $L$ into $A$.  By norm we will always mean reduced norm,
and the notation will be $\norm (x)$.  In particular, for $a \in K$,
$\norm(a) = a^n$.  
The usual property holds: $\norm(xy) = \norm(x)\norm(y)$ for
$x,y \in A$, whether or not $x$ and $y$ commute.  It follows that $\norm(ax) =
a^n \norm(x)$ for $a \in K$.

\par An element $a \in A$ is an \mathdef{integer}
if the monic irreducible polynomial of $a$ over $K$ has coefficients in $R$.  
Sums and products of commuting integers are integers, but, as we shall
see later, products of integers need not be integers.

\par  Suppose $A$ is central simple over $K$ of degree $n$.  Which elements of
$K$ are reduced norms of elements of $A$?  The answer is given by the theorem
of Hasse-Maass-Schilling (see \cite[p. 289]{reiner}):
\smallskip

\begin{thm}[Hasse-Maass-Schilling]\label{hms}\  An element $m$ of
$K$ is a reduced norm 
of an element of $A$ if and only
if $m$ is positive at every real place of $K$ at which $A$ is ramified.
\end{thm}

For convenience we will call this the HMS theorem.
Note that there is no condition at the complex places of $K$, at the
finite places of $K$, or at the real places of $K$ where $A$ does not ramify.

\par Suppose $m \in R$ and $m$ is a norm in $A$.  It need not happen that $m$ is the
norm of an integer of $A$.  We will call $m \in R$ an \mathdef{outlier}
if $m$ is a norm in $A$ but not the norm of an integer.  Equivalently,
$m$ is not the norm of an element of any maximal order.  We will be concerned with
the existence of, and properties of, outliers.

\par

If $K$ is a number field, we say $K$ is \mathdef{totally real} if $K_v$ is real at all
the infinite places $v$ of $K$.  If $K$ is totally real, 
$m$ in $K$ is \mathdef{totally positive}
if the real number $m_v$ is $>0$ at all the infinite places $v$
of $K$. 
The $m$ in $K$ for which $m_v > 0$ at the real places
of $A$ that ramify are, by Theorem \ref{hms}, 
the reduced norms of elements of $A$, and conversely.

We recast the identification of outliers in terms of Lemma \ref{splitting}.  Suppose
$A$ is central simple over $K$ of degree $n$, $R$ the ring of integers of $K$.

\begin{lemma}\label{outliers}\ 
Suppose $m \in R$ is a norm in $A$.  Then $m$ is {\bf not} an outlier if and only if
there is a monic irreducible polynomial $f(t) \in R[t]$ such that
(1) $f(0) = (-1)^nm,$ and (2) For each place $v$ of $K$, let $f(t) = \prod f_i(t)$ be the factorization
of $f(t)$ into irreducible monic factors in $K_v[t]$.  Then each $d_i = \degre(f_i)$
is a multiple of the local degree $n_v(A)$.
\end{lemma}

\begin{proof}
Let $L = K(\alpha)$ be the root field of $f$.  Then $[L : K] = n$ since
$f$ is irreducible, and $\alpha$ is an integer since $f \in R[t]$ is monic.  The first
condition says that the norm of $\alpha$ is $m$.  The second condition, by Lemma \ref{splitting},
says that $L$ splits $A$, and so $L$ embeds in $A$ since its dimension is $n$.  Then
the reduced norm of $\alpha$ is $m$.  The other direction of Lemma \ref{outliers} is clear.
\end{proof}

\begin{corollary}\label{invars}\ 
Suppose  $A$  and  $B$  are central simple algebras over $K$ of the same degree  $n$,
and that the local degree $n_v(A)$  divides the local degree $n_v(B)$  for all places  $v$  of $K$. 
Then any outlier $m$  of  $A$  is a priori also an outlier of  $B$.
In particular, if $B$ has no outliers then $A$ has no outliers.
\end{corollary}

\begin{proof} 
The polynomial requirements of Lemma \ref{outliers} for $B$ are more restrictive
than those for $A$.
\end{proof}

\par
In Section~\ref{max} we review maximal orders in the central simple
algebra $A$ and recall Eichler's condition.
In Section~\ref{hd} we prove that when Eichler's condition is
satisfied, then there are no outliers. In other words, if $A$ has
outliers then $A$ is a quaternion algebra over a totally real number
field $K$, and all real places of $K$ are ramified in $A$. However,
this condition is sufficient but not necessary:
there are definite quaternion algebras over totally real number fields
that have no outliers.
We remark that there is no logical relation between having outliers
and having a unique (up to conjugacy) maximal order; neither condition
implies the other.
In Section~\ref{quats} we study quaternion algebras over the the field
of rational numbers $\Q$. We particularly study definite quaternion algebras
ramified at a single finite prime. 
We write $A_r$ for the definite quaternion algebra over $\Q$ unramified
away from the places $\infty$ and $r$. 
We show, for example, that if $A_r$ has an outlier then it has an
outlier less than an explicit bound ($r^2/16$).

We give heuristic evidence that for infinitely many $r$,
$A_r$ has no outliers, as well as examples when chosen square free integers are
outliers.

\par

We are grateful to Joel Rosenberg for many discussions about the contents
of this paper, and for posing the questions which started us on this research.

\section{Maximal Orders} \label{max}

Let  $K$  be a number field, $R$ its ring of integers, and  $A$  a central simple algebra
over $K$.  
A subring $O$ of $R$ that contains $1$, is finitely generated as an
$R$ module, and that contains a basis of $A$ over $K$ is called an
\mathdef{order} of $A$.  Any order $O$ of $A$ is a projective $R$-module of rank
equal to $n$,  the degree of $A$ over $K$.  A \mathdef{maximal order} of $A$ is an order
which is maximal with respect to containment.  Maximal orders are isomorphic if and
only if they are conjugate, so we will speak of conjugacy classes of maximal orders.
All elements of a maximal order are \mathdef{integral} over $R$, and 
every integral element of $A$ is contained in some maximal order.




\par It is known that the number of maximal orders of $A$ , up to
conjugacy by an element of $A$, is finite. Let $\{O_1,\dots,O_t\}$
be a set of representatives.

\par For any given maximal order $O$ of $A$, let $I(O)$ be the group
of two-sided fractional ideals of $O$ modulo principal two-sided
fractional ideals. Set
$i(O) = \abs{O(I)}$.
It is known that each $i(O)$ is finite although the cardinalities
$i(O_1),\cdots i(O_t)$ may be distinct.
Their sum $c :=  i(O_1) + \cdots +i(O_t)$ turns out to be equal to 
the number of fractional left ideals of $O$ modulo principal
fractional left ideals for any maximal order $O$.

\par The terminology is that $t$ is called the type number and $c$ is
called the class number. We have just seen that the type number is at
most the class number.

\par Let $A$ be a CSA over a number field $K$. Consider the following three
conditions: 

\begin{enumerate}
\item $A$ is a quaternion algebra.
\item The field $K$ is totally real.
\item $A$ is ramified at every infinite place of $K$.
\end{enumerate}

\par It is customary to say that $A$ \emph{fails} the Eichler condition when
all three conditions hold.
For example, the quaternion algebra $A_p$ over $\Q$ ramified
at $p$ and $\infty$, and unramified away from those places,
fails the Eichler condition.

\par This description can be refined if the Eichler condition holds.
Assume now that $A$ satisfies the Eichler condition. Then the
$i(O_1),\cdots i(O_t)$
are all equal. In fact each $I(O_i)$ can be identified with an abelian
group $I=I(A)$, as do the types $T$ and the classes $C$. These three
abelian groups fit into an exact sequence
$$ 0 \to T \to C\to I\to 0.
$$
These three groups are related, via the reduced norm map.
to certain generalizations of the class group of the center $K$ of
$A$, by results of Eichler.

\par The group $C$ is isomorphic to the group $C'$ of fractional ideals of
$K$ modulo principal fractional ideals that can be generated by an
invertible element $a\in K$ that is positive at all infinite places of
$K$ that ramify in $A$.

\par Let $n$ be as usual the square root of $\dim_K(A)$.
If $\pp$ is a prime ideal of $K$ that is ramified in $A$, then at the
corresponding finite place $v$ of $K, A\otimes K_v = M_r{D'}$ for some
division algebra $D'$ over $K_v$ and for some $r$ dividing $n$.
The group $T$ is isomorphic to the the subgroup $T'$ of $C$ generated by
$nC$
and the class of $\pp^r$ for each finite prime $\pp$
(and note that this gives nothing new for the unramified primes since
$r=n$).

\par $I$ is isomorphic to the (abelian, finite) quotient group $C'/T'$.

\section{Higher Degree Central Simple Algebras} \label{hd}

Let $A$ be a central simple algebra of degree $n$ over the number field $K$.
The main result of this section is:

\bigskip
\par

\noindent{\bf Theorem A.} \label{thm:A}  {\it If $n>2$ then $A$ has no outliers.}

\par

\bigskip 

We need first a review of the proof of the HMS theorem in order to build a
variant that works for integers.  A first ingredient is:

\par

\bigskip

\noindent Krasner's Lemma: Let $v$ be a place of $K$, and
$f(t) = t^n + a_1 t^{n-1} + \cdots + a_n$ a separable irreducible polynomial
in $K_v[t]$.  If $g(t) \in K_v[t]$ is close enough to $f(t)$, then $g$ is
separable irreducible and $K_v[a]= K_v[b]$ where $a$ is a root of
$f(t)$ and $b$ is a root of $g(t)$.

\bigskip

\par

Eichler's proof of the HMS theorem goes as follows.  Let $R$ be the integers of $K$,
and $m \in R$ satisfying the required condition: $m$ is positive at all places
$v$ of $K$ which are real and ramified in $A$.  Let $S$ be the set of infinite places
of $K$ at which $A$ ramifies. Let $S'$ be a finite set of
finite primes of $K$, including those that ramify in $A$. We insist
that $S'$ be non-empty; if necessary, we include an irrelevant extra
prime where $A$ is unramified 
but where the polynomial constructed below is
irreducible.  We construct a polynomial

\begin{equation}
f(t) = t^n + c_1 t^{n-1} + \dots + (-1)^n m \in K[t]
\end{equation}

so that:
\par

\begin{itemize}

\item For each $v \in S'$, $c_i$ is close enough to an irreducible polynomial $f_v(t) = t^n + a_1 t^{n-1} + \dots +
(-1)^n m \in R_v[t]$ to guarantee $f$ is irreducible in $K_v[t]$.
There is such a polynomial \cite[XI, \S3, Lemma 2]{weil} but we don't show that here.

\item For each $v \in S$, $f$ is close to $f_v(t) = t^n + (-1)^n m$, i.e.
each $c_i$ is positive and close to $0$.\ (Note that if any such $v$
exists, then $n$ is necessarily even).
 This guarantees $f_v$ has no real roots.  If $A$
is not ramified at any infinite place of $K$, then this condition is vacuous.

\end{itemize}

\par

Since $S'$ is non-empty, $f$ is irreducible in $K[t]$.  Let $L = K(\alpha)$ where
$\alpha$ is a root of $f$; $[L : K] = n$.  The first condition on $f$ says that
$L$ splits $A$ at the finite primes, and the second condition guarantees that
$L$ splits $A$ at the ramified infinite places, since the root field of $f$ must
be complex.  The sign $(-1)^n$ guarantees that the norm from $L$ to $K$ of $\alpha$ is $m$.
Finally, since $L$ is a splitting field of degree $n$, then $L$ embeds in $A$ as a
maximal subfield, and the reduced norm of $\alpha$ is $m$.

\par

This is the proof rendered
by Eichler, and is the one presented in~\cite{reiner},
\cite{vigneras}, and~\cite{weil}.
Note that it made crucial use of the weak approximation
theorem.

\par

To go further, we use the strong approximation theorem~\cite[Corollary
2, page 70]{weil}, which better
suits our purposes. Let $w$ be a place of $K$ at which $A$ is unramified.
Then we can insist that the $c_i$ are in $R_v$ for all $v \ne w$.
We call this the strong proof of the HMS theorem.
We conclude:  any $m\in R$ which is positive at all real places of $K$ 
that ramify in $A$ is the reduced norm of an element $\alpha$ of $A$
that is integral at all places $v$ of $K$ not equal to $w$. 
So if $K$ has a complex place, or a real place that is not ramified in
$A$, then, taking this for $w$ shows that $A$ has no outliers.

\par 

\begin{lemma}\label{totdef}\ 
\ If $A$ has an outlier then, $K$ is totally real,
$A$ is totally definite, i.e.\ $A$ is ramified at all the real infinite places
of $K$.
\end{lemma}

\par

\begin{proof}
Let $m \in R$ be a norm in $A$.  If the conditions are not
satisfied, then $A$ must have an infinite place $w$ at which $A$ is
unramified.  We use this extra place  in the strong proof of
the HMS theorem.  Then the polynomial $f$ is in $R[t]$, $\alpha$ is an
integer, and $m$ is the norm of an integer.
\end{proof}
\par

\begin{lemma}\label{finite} \ If $A$ has an outlier then there is a finite place of
$K$ that ramifies in $A$.
\end{lemma}

\par

\begin{proof}
Suppose  $A$ is unramified at all finite places.
By Lemma \ref{totdef} , we may assume $n$ is even. Let m in $R$ be totally positive.
The polynomial  $t^n+m$  does the trick in the strong proof of HMS.
\end{proof}

\par

We finish the proof of Theorem~A.
By Lemma \ref{totdef} we may assume
that $K$ is totally real, $A$ is ramified at all real places, $n>2$ is even, and
$m$ is positive at all infinite places.
First we treat the finite places. By \cite[Ch. XI,\S3,Lemma~2]{weil},
for each finite place $v$, for any $n$, and for any nonzero $m$ in $K$ 
there exists a monic degree-$n$ irreducible polynomial $f(t) \in
K_v[t]$ with coefficients in $R_v$ such that $f(0) = (-1)^nm$.

Let $M_f$ denote the set of finite places of $K$ that ramify in $A$.
Note that $M_f$ is finite, and for each $v$ in $M_f$ we have $f_v(t)$
as required by Lemma \ref{outliers}, but we have not yet treated the
infinite places.

\par

For each $1 \le k \le n$ apply the Chinese remainder theorem to the coefficient
of $t^k$ in $f_v$ to get a monic polynomial $g(t) \in K[t]$ with $g(0) = m$ and
integral coefficients so that each localization $g_v(t)$ at
each $K_v$ is close enough to $f_v$ to be irreducible by Krasner's lemma.  We have
lifted the required polynomials at the finite primes, but the infinite places
are still at bay; there is yet no reason why $g(t)$ has only complex embeddings.

\par

Each $v \in M_f$ lies over some rational prime $p_v$. Let
$N = \prod_{v \in M_f} p_v$ be their product.

\par

Let $M_{\inff}$ be the set of real places of $K$ that ramify in $A$.
For any $v \in M_{\inff}$ we have a real polynomial $g_v$ which is positive
at $-\infty$ , $\infty$ and $0$ by construction. Therefore, there is some integer
multiple $M_v$ of $N$ so that $g_v(t) + M_vt^2$ is positive everywhere.
Let $M$ be the largest of the $M_v$. Furthermore by replacing $M$ by $N^kM$
for a sufficiently large $k$, we can insure by Krasner's lemma again
that $g_v(t) + N^kMt^2$ is irreducible at each $v \in M_f$.

\par

The polynomial $f(t) = g(t) + N^kM t^2 \in K[t]$
does the trick: it is monic of degree $n$, has no real roots, and for each place of 
$K$ that
ramifies in $A$, each irreducible factor of $f_v$ has degree a multiple of $n_v(A)$.
This finishes the proof of Theorem~A. \qed\hfill

\par

Note that the coefficient of $t^2$ was available for modification only
because $n > 2$.  For quaternion algebras, the coefficient of $t^2$ is
constant equal to $1$.  We get to that case next.

\par

\section{Quaternion Algebras}\label{quats}

We write $\Q$ for the field of rational numbers and $\Z$ for the ring of integers.

We consider definite quaternion algebras over $\Q$ with special attention to
$A_r$ = the definite quaternion algebra ramified at the prime $r$ and unramified
at all other finite primes.  Of course $A_r$ is also ramified at $\infty$, and
so at all infinite places.  The simplification here is that the integers which
are norms in $A_r$ are exactly the set of positive integers, and so the only
issue is whether they are norms of integers.  We now investigate how
this could happen.

\par

Let $m$ be a positive integer.  Let $f(t) = t^2 + bt + m$ with $b \in Z$.  Let $L = \Q(\alpha)$
with $f(\alpha) = 0$.  Then $L$ splits $A_r$ if and only if:

\begin{itemize}
\item $f$ is $r$-adically irreducible
\item $f$ has degree $2$ at $\infty$, i.e. $d = b^2 - 4m < 0$.
\end{itemize}

\par

When either of the conditions above hold, then $f$ is irreducible and \hbox{$[L : \Q] = 2$}.
When they both hold, $L$ embeds in $A_r$ by Lemma \ref{splitting},
$\norm (\alpha) = m$, and so $m$ is the norm
of an integer in $A_r$.  Moreover, $m$ is the norm of an integer
if and only if this search succeeds for some $b \in \Z$.  There are a finite number
of eligible $b$ by the last condition; $|b| < \sqrt{4m}$.  Furthermore, $b$ can
be assumed to be positive; if $\alpha$ is a root of $t^2 + bt + m$ then
$-\alpha$ is a root of $t^2 -bt + m$.  Of course $b = 0$ is legitimate as
a possibility.  We record this in:

\par

\begin{lemma}\label{poly}\  The positive integer $m$ is the norm of an
integer in $A_r$ if and only if there is a polynomial $f(t) = t^2 + bt + m$
satisfying the two conditions above for some $b \in \Z$.  $b$ need only be searched
in the range $0 \le b < \sqrt{4m}$.
\end{lemma}

\par

For polynomials of the right shape, they are irreducible $r$-adically if and only
if they are irreducible mod $r$.  So when is $m =2$ an outlier in $A_r$?  We illustrate
the search below, where we assume $r > 2$:

\par
\begin{equation}\label{two}
\begin{aligned}
b = 0 \quad  \quad d = -8  \\
b = 1 \quad  \quad d = -7  \\ 
b = 2 \quad  \quad d = -4  \\
\end{aligned}
\end{equation}

\par

\noindent Of course $-8$ is an $r$-adic square if and only if $-2$ is, and this happens
if and only if the Legendre symbol $\left( \frac{-2}{r}\right) = 1$.  Similarly, $-4$ is
a square if and only if $-1$ is.  For each of the three conditions in
\ref{two}, a random
prime $r$ satisfies it with probability $1/2$.  We conclude:

\begin{thm} \ The integer $2$ is an outlier in $A_r$ if and only if
\begin{equation}

\left(\frac{-2}{r}\right) =  \left(\frac{-7}{r}\right) = \left(\frac{-1}{r}\right) = 1

\end{equation}

\end{thm}

\noindent By considering the value of $r$ mod $56$, it follows from
the Dirichlet 
density theorem~\cite[Chapter VI, \S4, Theorem 2]{serre2}
that the set of primes $r$ for which this holds has density
$\frac{1}{8}$. In particular it is infinite.

\par

We do this once more to determine when $3$ is an outlier.  The data gives the following list:

\begin{equation}\label{three}
\begin{aligned}
b = 0 \quad  \quad d = -12  \\
b = 1 \quad  \quad d = -11  \\ 
b = 2 \quad  \quad d = -8  \\
b = 3 \quad  \quad d = -3  \\
\end{aligned}
\end{equation}

\noindent There is a redundancy; $-12$ is a square if and only if $-3$ is.  We conclude,
for $r > 3$:

\begin{thm} \ The integer $3$ is an outlier in $A_r$ if and only if
\begin{equation}

\left(\frac{-3}{r}\right) =  \left(\frac{-11}{r}\right) = \left(\frac{-2}{r}\right) = 1

\end{equation}

\noindent The set of primes $r$ for which this holds is infinite and has density
$\frac{1}{8}$.

\end{thm}

By similar analysis we get, for $r > 6$:

\begin{thm} \ The integer $6$ is an outlier in $A_r$ if and only if
\begin{equation}

\left(\frac{-2}{r}\right) =  \left(\frac{-3}{r}\right) = \left(\frac{-5}{r}\right) =
\left(\frac{-23}{r}\right) = 1

\end{equation}

\noindent The set of primes $r$ for which this holds is infinite and has density $\frac{1}{16}$

\end{thm}

\par

Suppose $6$ is an outlier for $A_r$.  It does not follow that $2$ and $3$ are outliers.
There may be integral $\alpha$ and $\beta$ with $\norm(\alpha) = 2$
and $\norm(\beta) = 3$, and then $\norm(\alpha \cdot \beta) = 6$.  It might happen
that for all such occurrences $\alpha$ and $\beta$ are in different maximal orders, and $\alpha \cdot \beta$ is not
integral.  When $6$ is minimal as an outlier, this is what had to happen.
This can be quantified; we state without proof:

\begin{thm} 
\ $A_r$ has the property that $2$ and $3$ are not outliers and $6$ is an
outlier if and only if $-2,-3,-5,-23$ are squares mod $r$ and either
\begin{equation}

\begin{aligned}

& \left(\frac{-1}{r}\right) = -1 {\text {\ or}} \\ 

& -1 {\text {\ is a square mod $r$  and $11$ and $7$ are non-squares mod $r$}}\\

\end{aligned}

\end{equation}

\noindent The set of primes $r$ for which this holds is infinite and has density
$5/128  = (1/16)(1/2 + 1/8)$.

\end{thm}

\par

We have not yet determined all outliers in $A_r$, nor have we answered
whether they are infinite when non-empty.  We need two results to
prepare for this.  We take on the second issue first.  Since the next
result holds more generally than for the $A_r$, we state it in full
generality.  In all of the following, the symbol $(a,b)$ stands for
the quaternion algebra over some ground field generated by $i$ and $j$
where $i^2 = a$, $j^2 = b$, $ij = -ji$.

\begin{thm} \label{psquare}
\ Let $A$ be a definite quaternion algebra over $\Q$ ramified at the finite
prime $r$.  If $m$ is a positive integer, then $m$ is an outlier for $A$ if
and only if $mr^2$ is also.

\end{thm}

\begin{proof}
For the easy direction: if $\norm(\alpha) = m$ with
$\alpha$ an integral element of $A$, then $\norm(r \cdot \alpha) = mr^2$.  We need to show
conversely that when $mr^2$ is the norm of an integer, so is $m$.

\par Let $O$ be any maximal order of $A$.  It is enough to show that whenever $mr^2$
is a norm of an element $\alpha$ of $O$, then $\alpha/r\in O$.  The completion of $O$ at $r$ is the
norm form of the unique quaternion algebra $D$ over $\Q_r$.  By \cite{serre}, $D$ has the form
$(a,r)$ where $a$ is an appropriate non-residue mod $r$.  When $r$ is odd, any non-residue
will do, whereas when $ r = 2$, $a = -3$ will do (in all cases, $\sqrt{a}$ determines the unique unramified
quadratic extension).  The norm form for this algebra is:

\begin{equation} \label{20}

F = x^2 - ay^2 - r(z^2 - aw^2)

\end{equation}

\par
Assume that $F(x,y,z,w) = mr^2.$ It follows that $x^2 - ay^2\equiv 0 \pmod{r}$.
As $a$ is a non-residue,
this forces $x$ and $y$ to be $\equiv 0 \pmod r$.  But then $x^2 -
ay^2\equiv 0 \pmod {r^2}$, and
so $r(z^2-aw^2)$ is $0$ mod $r^2$.  It follows that $z^2-aw^2\equiv 0
\pmod {r}$, so
that $z$ and $w$ are $0$ mod $r$.  Now all four coefficients $x,y,z,w$
of $\alpha$ are divisible by $r$.  Thus $\alpha / r$ is in
$O$ and has norm $m$.
We conclude
that whenever $\norm(\alpha)$ is $mr^2$ with $\alpha$ in $O$, then $\alpha / r$ is in
$O$ and has norm $m$.  Since this holds for all maximal orders, the lemma is established.
\end{proof}

\begin{corollary}\label{infinite}
\ With $A$ as in the Theorem \ref{psquare}, if the set of outliers for $A$ is non-empty,
then it is infinite; if $m$ is an outlier for $A$, then so is $m r^{2n}$ for any
positive integer $n$.
\end{corollary}

\begin{remark}\label{divp}
\ Corollary \ref{infinite} allows division by $r^2$, but not by $r$.  In fact,
if $m$ is an outlier for $A_r$ and relatively prime to $r$, then $mr$ is not an
outlier.  The polynomial $t^2 + mr$ is irreducible at $r$ by Eisenstein's criterion,
and also irreducible at infinity; it satisfies the requirements of Lemma \ref{outliers}.
\end{remark}

We need a bound up to which we can check for outliers not governed by Theorem \ref{psquare}.
We do this for $A_r$; the generalizations to definite quaternion algebras will be clear.
One more preliminary is necessary.

\begin{lemma}\label{residue}
\ Let $p>2$ be a prime and $m$ in $GF(p)$ nonzero. Then there exists $b$ in $GF(p)$
such that $b^2-4m$ is a nonsquare mod $p$.
\end{lemma}

\begin{proof}
Suppose not. Then, for every $b$, $b^2-4m$ is a square. But then
$(b^2-4m)-4m$ is a square and by induction $b^2-4mj$ is a square for all $j$.
By our hypotheses $4m$ is invertible in $GF(p)$ so all elements of $GF(p)$ 
are squares, contradiction.
\end{proof}

\par
We can now establish a bound for $A_r$.

\begin{thm} \label{bound}
\ Suppose $r > 2$ is prime and $m$ is a positive integer coprime to $r$.
Set $C(r) = r^2/16$.  If $m > C(r)$, then $m$ is not an outlier for $A_r$.
\end{thm}

\begin{proof}
By Lemma \ref{residue}, we choose an integer $b$ such that $b^2-4m$ is a nonsquare mod $r$.
We are free to assume of course that $b<r/2$. Set $f = t^2 + bt + m$, and $d = b^2 - 4m$.
One checks that the bounds on $b$ and $m$ say that $d < 0$, so $f$ is irreducible
at infinity.  Since $d$ is a non-residue at $r$, $f$ is also irreducible in $\Q_r[t]$.
Then $f$ satisfies the requirements of Lemma \ref{outliers}, and so $m$ is the norm
of an integer.
\end{proof}

\begin{remark}\label{effective}
\ Theorem \ref{bound} gives an effective strategy for finding all outliers in
$A_r$.  One checks all $m$ in the interval $[0,C(r)]$ using Lemma \ref{outliers}.
For $m > C(r)$: $m$ is not an outlier if $m$ is not divisible by $r$.  If $m$ is divisible
by $r$ to the first power, then $m$ is not an outlier by Remark~\ref{divp}.  If $m$
is divisible by higher powers of $r$, then successive uses of Theorem~\ref{psquare}
gets us to the case of first power or the range $[0,C(r)]$.
\end{remark}

\noindent Here is one case where all outliers can be determined.

\begin{corollary}\label{67}\ If $r = 67$, then the only outliers for $A_r$ are of form
$3 \cdot r^{2n}$, $n = 1,2,3 \dots$.
\end{corollary}

\begin{proof}
One checks in the range $[0,C(r)]$ that the only outlier
is $3$ using Lemma \ref{outliers} for each possible $m$.  Then Remark \ref{effective}
does the rest.
\end{proof}

\begin{remark}\label{hypotheses}
\ Note that this corollary says that division by $r^2$ is not always possible
when $r$ is not a ramified prime.  In $A_{67}$, $12 = 3 \cdot 2^2$ is the norm of
an integer, but $3$ is not, so division by the square of the unramified prime $2$ is not
possible.
\end{remark}

\par
An effective bound for more general definite quaternion algebras is not difficult.
Suppose $A$ is a quaternion algebra central over $\Q$ ramified at infinity and the
finite primes comprising a set $S$.  Let $C$ be the product of the finite ramified
primes of $A$, and $M = C^2/16$.  Then

\begin{thm}\label{general}\ $M$ is an effective bound for determining all the outliers for $A$.
\end{thm}

\noindent The proof is exactly as in Theorem \ref{bound} and Remark \ref{effective}.

\par The symbol $ B = (-58,-17)$ over $\Q$ is ramified at infinity and the
finite primes $S =  \{2,17,29\}$.  Using Theorem \ref{general} one can show:

\begin{corollary}\label{ten}
\ The outliers for $B$ are the set $\{10 r^{2n} : n = 1,2,3,\dots, \ r $ a product of
elements of $S \}$.
\end{corollary}

\par

The minimal outlier of $B$ is $10$.  Therefore, there are integers
$\alpha$ and $\beta$ in $B$ with $\norm(\alpha) = 2$ and $\norm(\beta) = 5$.
Whenever this happens, the product $\alpha \beta$ is not integral.

\par
The appearance of $6$ and $10$ in this context is general, as seen in the next theorem.

\begin{thm}\label{general2}
\ Let $m$ be a positive integer that is not a square.  
Then there are infinitely many
primes $r$ such that $\{m\cdot r^{2n} : n\ge 1\}$ are outliers for $A_r$.
\end{thm}

\begin{proof}
For $b$ in the range $[0,C]$, $C = {\sqrt{4m}}$, and 
$d = b^2 - 4m$, we must have the Legendre symbol $\left(\frac{d}{r}\right)$ equal
to $1$; the Cebotarev density theorem says there are infinitely many such primes
$r$.  In fact their density is $1/2^s$ for some appropriate integer $s$.
\end{proof}

\section {Open Questions}

\par

We begin with:
\par

\bigskip

\noindent Are there infinitely many rational primes  $r$  such that $A_r$ has no outliers?

\bigskip

\par

Heuristically, the answer is yes.  Computer searches for small bounds show
that $A_r$ has no outliers a little more than half the time.
\par

We have seen that  the set of primes $r$ for which $m=2$ is an outlier for $A_r$ has density
$1/8$. Similarly, the set of primes for which $m= 3$ is an outlier for $A_r$ has
density $1/8$. 
%
%
Adding together these probabilities for small $m$ appears
to give something like density $0.7$; this is roughly the probability that
neither $2$ nor $3$ is an outlier.  However for large $m$, the density of primes $r$
for which $m$ is an outlier in $A_r$ should be something like $2^{-c \sqrt{4m}}$,
since we are asking that the  floor of $\sqrt{4m}) + 1$ numbers are all squares
$r$-adically and some constant $c$ is required because these numbers may not be
linearly independent in 
$\Q^*/(\Q^*)^2$. However the sum

$\sum_{m>0} 2^{-c \sqrt{4m}}$ converges.  Therefore we cannot distinguish whether our 
set is finite or infinite.

\par

Another interesting question concerns totally definite
quaternion algebras over totally real number fields.
Do they have outliers? Sometimes? Often?
\par 

We have worked out only one example.  Let $B = (-1,-7)_K$ where $K$ is the real
subfield of seventh roots of unity.  Then $B$ has no outliers.  The argument is 
technical, so we will not reproduce it here;  
it requires a detailed study of units, totally positive units, class
number, and the establishment of a bound as in Theorem \ref{bound}; the bound is
$1792$. However,
when $K = \Q({\sqrt 2})$, the same algebra tensored up to $K$ does have outliers.
Thus, restriction maps may or may not preserve the property of having no outliers.

\par

On the other hand, let $A$ be the algebra $(-1,-67)$ over $\Q$; by Corollary~\ref{67},
$3$ is an outlier for~$A$. 
If $K = \Q(\sqrt{67})$, then $A \otimes_{\Q} K$ is ramified
at only the infinite places of $K$, and so by Lemma \ref{finite} has no outliers.
Thus the restriction map may also fail to preserve the property of having outliers.

\par
The last remark can be generalized.  From Lemma \ref{finite}, if $A$
is a quaternion algebra over $\Q$, then there is a real quadratic field $K$
so that:

\begin{itemize}
\item $A \otimes_{\Q} K $ is a division ring
\item $A \otimes_{\Q} K $ has no outliers.
\end{itemize}

\section{Application to supersingular elliptic curves and surfaces}
We review the connection between supersingular
elliptic curves in characteristic $r$ and maximal orders in $A_r$, 
where $A_r$, is the definite quaternion algebra ramified at $\infty$ and $r$ and
unramified away from these places.

Let $E$ be a supersingular elliptic curve defined over $\Omega$,
an algebraic closure $\Omega$ of $GF(r)$,
and write $\End(E)$ for its endomorphism ring.
Then $\End(E)\otimes_{\Z}\Q$ is isomorphic to $A_r$.
Under this isomorphism, $\End(E)$ is a maximal order in $A_r$, 
and, conversely, any maximal order $M$ of $A_r$ is isomorphic to $\End(E)$ for
some $E$.

Furthermore the norm of an endomorphism $\phi:E\to E$ is, under this
isomorphism, equal to the reduced norm of the corresponding $m\in M$.

The statement ``$m$ is an outlier for $A_r$'' translates to:
no supersingular elliptic elliptic curve defined over 
$\Omega$ has an endomorphism of degree $m$.

So we see, for example, that for every integer $m>1$ 
there are infinitely many primes $p$ such that 
no supersingular elliptic curve defined over 
$\Omega$ has an endomorphism of degree $m$.

Next we turn to products of supersingular elliptic curves.
\begin{corollary} \label{here}\ 
Let $E$ be a supersingular elliptic curve defined over an algebraic closure
of $GF(p)$, and set $A = E^g$ for $g\ge2$ an integer.
Then the abelian variety $A$ has an endomorphism of
degree $m$ for every positive integer $m$.
\end{corollary}
\begin{remark}\ Here we are considering all endomomorphisms of $A$, 
not just those that preserve the obvious principal polariziation.
\end{remark}
\begin{proof}[Proof 1] $\End(A)$ (which happens to equal $S = \Mat_g(M)$)
is a maximal order in the central simple algebra $\Mat_g(A_r)$ of
dimension $4g^2$ over $\Q$. Eichler's methods, as outlined in
Section~\ref{max}, imply that $S$ is the unique maximal order up to
conjugacy since $g>1$.
Therefore, by Theorem~A, 
$m$ is a reduced norm of an element $\alpha\in \End(S)$.
However, reduced norm is in this case equal to the degree of the map
$\alpha$.
\end{proof}

\par
Our second proof of Corollary~\ref{here} uses a well-known theorem of
Deligne, which we state below. As 
we have not found an adequate proof in the literature, for the
reader's convenience we include one in the next section.

\begin{thm}[Deligne] \label{thm:deligne}\ 
Let $p$ be a prime and let $n\ge2$ be an integer. If $E_1,\cdots,E_n$ and 
$F_1,\cdots,F_n$ are supersingular elliptic curves defined over 
an algebraic closure of $GF(p)$, then
$$E_1\times\cdots\times E_n\cong F_1\times\cdots\times F_n
$$
\end{thm}
\begin{proof}[Proof 2 of Corollary~\ref{here}]
It turns out it is enough, using Deligne's result, to show that for
each rational prime $\ell$ that $A=E^n$ has an endomorphism of degree $\ell$.

There does exist for each $\ell$ an isogeny of supersingular elliptic
curves $\phi:E\to E'$ of degree $\ell$ (for $\ell=p$, the Frobenius
has degree $p$ and for $\ell \ne p$ mod out by any subgroup $H$ of
order $\ell$). However, Deligne's theorem gives an
isomorphism
$$
\psi: E^n \cong E'\times E^{n-1}.
$$
Thus the composite $\psi^{-1}\circ (\phi\times \id^{n-1})$ furnishes the
desired endomorphism of degree $\ell$.
\end{proof}

We are grateful to Bruce Jordan for suggesting the second proof of Corollary~\ref{here}.
\section{Proof of Deligne's theorem} \label{sec:deligne}
In this section we prove Theorem~\ref{thm:deligne}.
For $p$ a prime, let $\Omega$ denote an algebraic closure of $GF(p)$.


\begin{remark} 1. It would suffice by induction to prove the theorem for $n=2$
(although we will not use this remark).
2. It will suffice to show (by transitivity of
  isomorphism)
that $F_1\times\cdots\times F_n \cong E^n$ for some particular
supersingular
elliptic curve $E$ defined over $\Omega$.
\end{remark}
The remainder of this section is devoted to the proof of
Theorem~\ref{thm:deligne}.

Note that $\Delta = \End(E)$ is a maximal order in the quaternion algebra 
$
A_p = \Delta  \otimes_{\Z} \Q$. The left $\Delta$-module
$\Hom(F_1\times \cdots \times F_n,E)$
being a projective module of rank $n\ge 2$ is free by \cite{reiner}[Corollary 35.11 (iv)]
(By the results of Section~\ref{max} since the Eichler condition
holds for $M_n(\Delta)$, and since visibly any ray class field over
$\Q$ is trivial).
This is the key point in the proof.

\par

Let $\phi_1,\cdots ,\phi_n$ be a basis. 
The freeness means that  
any homomorphism $\psi$ from $F_1\times \cdots \times F_n$ to 
$E$ is uniquely a sum
$$
\psi = \delta_1\circ \phi_1 + \cdots + \delta_n \circ \phi_n
$$ for some
$\delta_1,\cdots, \delta_n$ in $\Delta$, noting that the $\Delta$
action on $\Hom(A,E)$, 
is composition of functions. Setting $\Phi = (\phi_1,\cdots,\phi_n)$,
we have constructed a homomorphism 
$$
\Phi: F_1\times \cdots \times F_n \to E^n,
$$
and to finish the proof of the theorem it will suffice to
prove that $\Phi$ is an isomorphism.

Let $K$ be the kernel of $\Phi$. If $\Phi$ is not an isomorphism, then
$K$ is nontrivial, and therefore some projection $\pi_i(K)$ is 
nontrivial in $F_i$. Let $\rho: F_i\to E$ be a homomorphism, and set 
$\psi: F_1\times \cdots \times F_n \to E$ to be $\psi(x_1,\cdots,x_n)
= \rho(x_i)$. It follows that $\rho$ and therefore any
homomorphism from $F_i$ to $E$ must kill $\pi_i(K)$.

\begin{lemma} \label{lem1}
There is a supersingular elliptic curve $E_0$ defined over $GF(p).$
\end{lemma}

\begin{proof}
Let $E$ be a supersingular elliptic curve defined over
$\Omega$. Then there is only one isogeny of order $p$ from $E$ to another
elliptic curve, namely the Frobenius isogeny $Fr:E\to E^{(p)}$.
It follows that $E$ has an endomorphism of degree $p$ if and only if
$E$ is defined over $GF(p)$.

Consider now the element $\sqrt{-p}$ in the quadratic number field
$L=\Q(\sqrt{-p})$. It has norm $p$ and is integral. As $L$ splits
$A_p$, it embeds in $A_r$. So $\Z[\sqrt{-p}]$ a fortiori embeds and
thus is contained in a maximal order $O$ of $A_r$. 
The usual norm on $\Z[\sqrt{-p}]$ is equal to the restriction of the
reduced norm under the embedding.
In the correspondence between maximal orders of $A_p$ and 
supersingular elliptic curves over $\Omega$, the elliptic curve
corresponding to $O$ is thus defined over $GF(p)$.
\end{proof}

The proof of Deligne's theorem will be completed by the following
lemma.

\begin{lemma} \label{lemma:int} Let $E$ and $F$ be supersingular elliptic curves over
$\Omega$. Then the intersection, as subgroup schemes of $\ker(\phi)$
as $\phi$ ranges over $\Hom(E,F)$, is trivial.
\end{lemma}

\begin{remark} 1. It will suffice to find a collection of isogenies
from $E$ to $F$ whose degrees are coprime.
2. It will suffice to prove the lemma for a fixed elliptic curve
$E_0$ (and $F$ varying), then precomposing with the  
dual isogenies  from $E_0$ to $E$ coming from 1.
\end{remark}

\begin{proof}

First of all we know that the $\Hom(E,F)$ is non-zero.
If $O=\End(E)$ is the maximal order of $\End(E)\otimes_{\Z}\Q$ 
corresponding to $E$, then $O$ has an
ideal whose right order is equal to the maximal order corresponding to
$F$ and this furnishes a non-zero isogeny. 
So $\Hom(E,F)$ is a finitely generated projective left module over $O$
(and not the zero module).

Let $K$ denote the intersection
(as subgroup schemes) of  $\ker(\phi)$ as $\phi$ ranges over $\Hom(E,F)$.
We have just showed that $K$ is finite. Among all isogenies from $E$ to $F$,
let $\phi$ be one of least degree.

Let $\ell$ be a prime dividing $\deg(\phi)$, hence also the order of $K$.

We first treat the somewhat easier case $\ell\ne p$.
If $\phi(E[\ell])=0$ then $(1/\ell)\phi$ is a non-zero isogeny
from $E$ to $F$ of smaller degree, contradiction.
Thus $W = Ker(\phi) \cap E[\ell]$ is one-dimensional.
However, $\End(E)$ acts transitively on the one-dimensional subspaces
of $E[\ell]$. Thus there is a $\sigma$ in $\End(E)$ that does not fix
$W$.
Then $\phi + \phi\circ\sigma$ is an isogeny from $E$ to $F$ of order
prime to $\ell$.

We finish the proof of the lemma in the case $\ell=p$.

It is enough by transitivity to assume (by Lemma~\ref{lem1}) that $E=E_0$
is defined over $GF(p).$
Assume that every isogeny from $E_0$ to $E$ has degree divisible by $p$.
Let $\phi:E_0\to E$ be the nonzero isogeny of least degree. If $\phi$ has
degree
divisible by $p$ then $\phi$ factors through the Frobenius. $\phi =
\psi\circ Fr$ for some $\psi:E_0^{(p)}\to E.$
But since $E_0^{(p)}=E_0$ then $\psi:E_0\to E$ is an isogeny 
of degree smaller than $\deg(\phi)$.
\end{proof}

This finishes the proof of Lemma~\ref{lemma:int} and of Theorem~\ref{thm:deligne}.

\end{document}